\documentclass[11pt]{article}
\usepackage{amsmath,amsfonts,amsthm,amscd,amssymb,graphicx}

\numberwithin{equation}{section}

\newtheorem{theorem}{Theorem}[section]

\newtheorem{lemma}[theorem]{Lemma}

\newtheorem{remark}[theorem]{Remark}

\def\bR  {\mathbb{R}}
\def\del{\partial }
\def \la {\langle_x}
\def \ra {\rangle}

\def\bR  {\mathbb{R}}

\def\del{\partial }
\def \la {\langle}
\def \ra {\rangle_x}

\date{June 7, 2018}

\begin{document}

\title{On the Extension of Onsager's Conjecture for General Conservation Laws}
\author{Claude Bardos\footnotemark[1] \quad Piotr Gwiazda\footnotemark[2] \quad Agnieszka \'Swierczewska-Gwiazda\footnotemark[3] \quad Edriss S.\, Titi\footnotemark[4] \, \and \,  Emil Wiedemann\footnotemark[5]}
\maketitle

\begin{abstract}

The aim of  this work   is to extend and prove the Onsager conjecture for a class of conservation laws that possess generalized entropy. One of the main findings of this work is the ``universality" of the Onsager exponent, $\alpha > 1/3$, concerning the regularity of the solutions, say in $C^{0,\alpha}$, that guarantees the conservation of the generalized entropy; regardless of the structure of the genuine nonlinearity in the underlying system.

\end{abstract}
{\bf Keywords:} {Onsager's conjecture, conservation laws, conservation of entropy.} \\
 {\bf MSC Subject Classifications:} {35Q31.}

\tableofcontents

\renewcommand{\thefootnote}{\fnsymbol{footnote}}

\footnotetext[1]{%
Laboratoire J.-L. Lions, BP187, 75252 Paris Cedex 05, France. Email:
claude.bardos@gmail.com}

\footnotetext[2]{%
Institute of Mathematics, Polish Academy of Sciences, \'Sniadeckich 8, 00-656 Warszawa, Poland.  Email: pgwiazda@mimuw.edu.pl}

\footnotetext[3]{%
Agnieszka \'Swierczewska-Gwiazda:  Institute of Applied Mathematics and Mechanics, University of Warsaw, Banacha 2, 02-097 Warszawa, Poland.  Email: aswiercz@mimuw.edu.pl}

\footnotetext[4]{%
Department of Mathematics,
                 Texas A\&M University, 3368 TAMU,
                 College Station, TX 77843-3368, USA. Also, Department of Computer Science and Applied Mathematics, The Weizmann Institute of Science, Rehovot 76100, Israel. Email: titi@math.tamu.edu \, and \, edriss.titi@weizmann.ac.il}

\footnotetext[5]{%
 Institute of Applied Mathematics, Leibniz University Hannover, Welfengarten~1, 30167 Hannover, Germany. Email: wiedemann@ifam.uni-hannover.de}

\section{Introduction}

In this work  we aim at extending and proving  the Onsager conjecture for a class of conservation laws that admit a  generalized entropy. Roughly speaking, the Onsager conjecture \cite{ON} states that weak solutions of the three-dimensional Euler equations of inviscid incompressible flows conserve energy if the velocity field $u\in C^{0,\alpha}\,,$ for $\alpha > \frac13\,,$ and that the critical exponent $\alpha =\frac13$ is sharp. This conjecture has been the subject of intensive investigation for the last two decades. The sufficient condition direction was proved by Eyink \cite{GEY} for the case when  $\alpha > \frac12\,.$ Later, a complete proof was established by Constantin, E and Titi \cite{CET} (see also \cite{CCFS}) under slightly weaker regularity assumptions on the solution which involve a similar exponent $\alpha > \frac13\,.$  Duchon and Robert \cite{DuchonRobert} have shown, under similar sufficient conditions to those in \cite{CET}, a local version of the conservation of energy. It is worth mentioning that the above results are established in the absence of physical boundaries, i.e., periodic boundary conditions or the whole space. However, due to the well recognized  dominant role of the boundary in the generation of turbulence (cf.\, \cite{BT} and references therein) it seems very reasonable  to investigate the analogue of the Onsager conjecture in bounded domains. Indeed, for the three-dimensional Euler equations in a smooth bounded domain $\Omega$, subject to no-normal flow (slip) boundary conditions, it has been shown in \cite{BT2} that a weak solution conserves the energy provided the velocity field  $u\in C^{0,\alpha}(\overline{\Omega})\,,$ for $\alpha > \frac13\,,$    (see also \cite{RRS} for the case of the upper-half space under stronger conditions on the pressure term). A local version, analogue to that of \cite{DuchonRobert}, was established recently in \cite{BTW} under slightly weaker conditions to those in \cite{BT2}, but at the expense of additional sufficient conditions concerning the vanishing behavior of the energy flux near the boundary.

Showing the sharpness of the  exponent $\alpha = \frac13$ in Onsager's conjecture turns out to be much more subtle. This direction has been underlined  by a series of contributions
(cf.\, Isett \cite{Isett}, Buckmaster, De Lellis ,  Sz\'ekelyhidi and  Vicol  \cite{BDSV} and references therein) where  weak solutions,   $u\in C^{0,\alpha}\,,$ with $\alpha <\frac 13$, that  dissipate energy were constructed using the convex integration machinery. Notice, however,  that there exists a family of weak solutions to the three-dimensional Euler equations, that are not  more regular than $L^2$,  and which conserve the energy, cf. \cite{BT-Shear}.

It is most natural to ask whether the analogue of the Onsager conjecture is valid for other systems of conservation laws. Indeed, there has been some intensive  recent work  extending  the Onsager conjecture for other physical  systems, in the absence of physical boundaries, see, e.g.,  \cite{AW,DrEy,FGSW,GMS,Yu} and references therein. In this paper we consider systems of conservation laws with physical boundaries. We use the approach of \cite{DuchonRobert}, as it has been outlined and extended in \cite{BTW} in the presence of physical boundaries, to establish the local conservation of ``generalized entropies" (conserved quantities, which are not necessarily convex) for systems of conservation laws which possess such generalized entropies. This is accomplished provided the underlying weak solutions are locally in $C^{0,\alpha}\,,$ for $\alpha > \frac13\,.$ One of the  primary findings of this work  is the universality of the Onsager critical exponent $\alpha = \frac13\,,$ regardless of the structure of the genuine nonlinearity in the underlying system. Notably, in a forthcoming paper \cite{BGSTW} we will show the extension of these results, and provide additional explicit physical examples,  using more delicate harmonic analysis tools and function spaces. In particular, we will show similar results  under slightly weaker regularity assumptions of the solutions which are required to belong to some ``exotic" function spaces with exponent $\alpha = \frac13\,.$

\section{Local and global generalized entropies conservation}
In this section we state and prove our main result. In the first subsection we establish  the local entropy conservation for any weak solution that belongs to the  H\"older space $C^{0,\alpha}$ with $\alpha > 1/3\,.$ In the second subsection we state a fundamental Lemma concerning nonlinear commutator estimate of Friedrichs mollifier, and in the last subsection we state additional sufficient conditions for establishing the global entropy conservation.

\subsection{Local entropy conservation}
Let  $Q\subset\bR^{d+1}$ be an open set, and consider in $Q$ the following system of conservation laws:
\begin{equation}
\sum_{0\le i \le d} \del_{x_i} A_i(u)=0\,, \label{eq1}
\end{equation}
where $x\mapsto u(x) $ is the unknown vector field   defined in $Q $ with values in  an open convex set  $\mathcal M\subset\subset   \bR^{k}\,, $ while the vectors  $A_i\,,i=0,1,\cdots,d\,$, are $C^2$ vector-valued functions defined in   $\mathcal M$ with values in $\bR^{l}$, where $A_i^j$, for $ j= 1, \cdots, l\,,$ denotes the $j$-$th$ component of $A_i\,$.

\begin{theorem}\label{new1}
Suppose that $B:\mathcal M\mapsto \mathcal L (\bR^{l};\bR) \,$ is a $C^1$ map, represented by a row vector $ B(u)=(b_1(u),b_2(u),\cdots, b_{l}(u))$, defines a
  {\it generalized entropy} , i.e., for every $ i=0,1,\cdots, l$ there exists a $C^2$ flux $q_i:\mathcal M\mapsto\bR$ such that one has:
\begin{equation}\label{g-entropy}
B(u) \cdot\nabla_u A_i(u)=\nabla_u q_i(u), \, \hbox{for}\, \,i=0,1,\cdots,d\,.
\end{equation}
Suppose that $u$ is a weak solution of (\ref{eq1}). Moreover, supposed that for every $\mathcal K \subset\subset Q$ there exists $\alpha >\frac13\,,$  which might depend on $\mathcal K$, such that $u\in C^{0,\alpha}(\mathcal K)$\,.  Then the following equation holds in $\mathcal D'(Q)$
\begin{equation}\label{conservation}
\sum_{0\le i\le d} \del_{x_i} q_i(u)=0\,.
\end{equation}
\end{theorem}
\begin{remark}
(i) Observe that since $\mathcal M$ is assumed to be an open convex, and hence simply-connected, set then the generalized entropy condition (\ref{g-entropy}) is  equivalent to the relation:
\begin{equation}
\sum_{1\le j \le l} \frac{\del b_j}{\del u_\beta }\frac{\del A_i^j}{\del u_\alpha}=\sum_{1\le j \le l} \frac{\del b_j}{\del u_\alpha }\frac{\del A_i^j}{\del u_\beta}\,, \quad
\hbox{for all} \,\, \alpha, \beta =1,2,\ldots,l\,, \,\,   \hbox{and} \,\, i=1,2,\ldots,d\,.
\end{equation}
For more  details about entropy cf. \cite{DA} , Chapters 3 and 5.
\noindent (ii) Notice that equation (\ref{conservation}) is, in  a sense, the analogue of the local conservation of energy for  the three-dimensional incompressible Euler equations as presented in \cite{DuchonRobert} (see also \cite{BTW}).
\end{remark}

The proof of Theorem \ref{new1} is an extension of the ideas introduced in  Bardos, Titi  and Wiedemann \cite{BTW}.
What has to be proven is that for any $\phi\in \mathcal D(Q)$ one has:
\begin{equation}
\sum_{0\le i\le d}\la \del_{x_i}\phi ,q_i(u) \ra =0\,.
\end{equation}
The support of $\phi$ being given one introduces three  open sets $Q_i$ such that  $\hbox {supp} \, \phi \subset \subset Q_1 \subset \subset Q_2 \subset \subset Q_3\subset\subset Q$
and such   for   $i=1,2,$ one has
$$
\sup_{x\in \Omega_i, y\in \bR^{d+1}\backslash \Omega_{i+1} }|x-y|>\epsilon_0\,,
$$
for some $\epsilon_0 >0$ that depends on the support of $\phi$ and   $Q\,.$

Then one introduces a $C^\infty $ cutoff function $I:\bR^{d+1} \mapsto [0,1]\,,$ which is zero outside $Q_3$ and is equal to $1$ on $Q_2$\,. For any distribution $T\in \mathcal D'(Q)$ one denotes by $\overline T\in \mathcal{D}'(\bR^{d+1})$ the distribution defined for every $\psi \in \mathcal D(\bR^{d+1})$ by the formula
\begin{equation}
\la \psi,\overline T\ra= \la I\psi ,T\ra\,.
\end{equation}
Eventually,  we use standard $C^\infty (\bR^{d+1})$ radially symmetric compactly supported Friedrichs mollifier $x\mapsto \rho_\epsilon(x)\,,$ with support inside the ball of radius $\epsilon >0$. For any distribution $T\in \mathcal D' (\bR^{d+1})$, we define the distribtion $T^\epsilon := T*\rho_\epsilon$. Next, we fix  $\epsilon\in (0,\frac{\epsilon_0}2)\,,$ which we will eventually let it tend to zero.

Observe that from (\ref{eq1}) one infers that
\begin{equation} \label{eq2}
\sum_{0\le i \le d} \del_{x_i} A_i(\overline{u})=0\,, \quad \hbox{in} \quad  \mathcal D(Q_2)\,.
\end{equation}
Notice that for any  $\Psi\in C_{\mathrm c}^2(Q_1;\mathcal L(\bR^l,\bR))$ one has that $\Psi^\epsilon \in \mathcal D(Q_2;\mathcal L(\bR^l,\bR))$. Therefore, as a result of (\ref{eq2}) one has
 \begin{equation}\label{Extension1}
 \begin{aligned}
&0=  \la \Psi^\epsilon ,\sum_{0\le i\le d}  \del_{x_i} A_i(\overline{u})\ra= -\sum_{0\le i\le d}\la (\del_{x_i}\Psi)^\epsilon ,    A_i(\overline {u})  \ra =
\sum_{0\le i\le d} \int_{\bR_x^{d+1}} \del_{x_i}\Psi (x) \cdot (A_i(\overline {u}))^\epsilon (x)\, dx
\\
&
=- \sum_{0\le i\le d}  \int_{\bR_x^{d+1}}    \del_{x_i} \Psi (x) \cdot  A_i((\overline u )^\epsilon)(x)\, dx
- \sum_{0\le i\le d} \int_{\bR_x^{d+1}}   \del_{x_i} \Psi (x) \cdot \Big( (A_i(\overline u))^\epsilon (x)   -A_i((\overline u)^\epsilon) (x) \Big) \, dx \,.
\end{aligned}
\end{equation}

Now,  we replace  $\Psi$,   in (\ref{Extension1}),  by 	$\phi B ((\overline u)^\epsilon) \in C_{\mathrm c}^2(Q_1; \mathcal L (\bR^l,\bR))$. Thus, the right-hand side  of (\ref{Extension1}) is  the sum of two terms:
\begin{equation}
J_\epsilon =    - \sum_{0\le i\le d}  \int_{\bR_x^{d+1}}     \del_{x_i}(\phi B((\overline u)^\epsilon) (x) \cdot  A_i((\overline u)^\epsilon ))(x) \, dx \,,
\end{equation}
and
\begin{equation}
K_\epsilon=\sum_{0\le i\le d}  \int_{\bR_x^{d+1}}   \del_{x_i}(\phi B ((\overline u)^\epsilon)(x)  \cdot \Big(A_i((\overline u)^\epsilon )(x) - ( A_i(\overline u))^\epsilon (x) \Big)\, dx\,.
\end{equation}
Thanks to (\ref{g-entropy}) one has for $J_\epsilon$:
\begin{equation}\label{Jepsilon}
\begin{aligned}
&J_\epsilon=- \sum_{0\le i\le d}  \int_{\bR_x^{d+1}}    \del_{x_i}(\phi B( (\overline u)^\epsilon))(x) \cdot  A_i((\overline u)^\epsilon )(x) \, dx \\
&= \sum_{0\le i\le d} \int_{\bR_x^{d+1}}    (\phi B( (\overline u)^\epsilon))(x)  \cdot \del_{x_i}  A_i((\overline u)^\epsilon )(x) \, dx =\\
& \sum_{0\le i\le d} \int_{\bR^d}  \Big[ \phi (B( \eta (\overline u)^\epsilon)) \cdot ( \nabla_u A_i((\overline u)^\epsilon )\cdot \del_{x_i}(\overline u)^\epsilon \Big ]dx
 = \\
 &\sum_{0\le i\le d}  \int_{\bR^d}  \phi  (x)  \del_{x_i} q_i((\overline u)^\epsilon(x)) dx =
-\sum_{0\le i\le d}  \int_{\bR^d}  \del_{x_i}\phi  (x)   q_i((\overline u)^\epsilon(x)) dx \,.
\end{aligned}
\end{equation}
Since $u\in C^{0,\alpha}(Q_2)\,,q_i\in C^2(\mathcal M)$  and $\phi\in \mathcal D(Q_1)\,,$  then by virtue of the Lebesque Dominant Convergence theorem this last term in (\ref{Jepsilon}) converges, when $\epsilon\rightarrow 0\,,$ to
$$
-\sum_{0\le i\le d} \int_Q \del_{x_i} \phi(x) q_i(u)dx=\sum_{0\le i\le d}\la \phi , \del_{x_i} q_i(u) \ra\,.
$$

To complete the  proof one has to show that for $\alpha>\frac13$ the term $K_\epsilon$ converges to $0\,,$ as $\epsilon\rightarrow 0\,.$
Obviously, one has:
\begin{equation}\label{bad-exponent}
\| \del_{x_i}(\phi B((\overline u)^\epsilon))\|_{L^\infty(Q_1)} \le   \|\phi\|_{C^1(Q_1)} \| B\|_{C^0(\mathcal M)} +
C \|\phi\|_{L^\infty(Q_1)}\|B\|_{C^2(\mathcal M)} \|u\|_{C^{0,\alpha}(Q_2)}\epsilon^{\alpha-1}\,.
\end{equation}
Therefore the proof is completed by virtue of the following estimate.
\begin{equation}\label{basic0}
\|((A_i((\overline u)^\epsilon ) - (A_i(\overline u))^\epsilon)\|_{L^\infty(Q_1)} \le C \|A_i\|_{C^2(\mathcal M)} \|u\|^2_{C^{0,\alpha}(Q_2)}\epsilon^{2\alpha}\,,
\end{equation}
which we will establish in the next section.

\subsection{Nonlinear commutator estimate for Friedrichs mollifier}
Since $\overline u $ in (\ref{basic0}) belongs to $C^{0,\alpha}(\bR^{d+1})$, with the H\"older exponent $\alpha$ that corresponds to the compact set $\overline{Q_3}$, and is  compactly supported in $Q_3$, estimate (\ref{basic0}) follows from the following more ``general lemma" concerning nonlinear commutator estimate for the Friedrichs mollifier (see also \cite{GMS}). Notably, the estimate below, which generalizes those established for quadratic nonlinearities in \cite{CET}  and \cite{BT2},  shows that regardless of the structure of the nonlinearity one always obtains the same exponent for $\epsilon$, i.e. $2\alpha$ in (\ref{basic0}). Combining (\ref{basic0}) with  (\ref{bad-exponent}) implies the ``universality" of the Onsager exponent, i.e., $\alpha > 1/3$, for conservation laws.

\begin{lemma}\label{basic-lemma} For any $F\in C^2(\mathcal M; \bR^l)$  and for any compactly supported function  $v \in C_{\mathrm {c}}^{0,\alpha}(\bR^{d+1}; \mathcal M)$  one has:
\begin{equation}\label{basic1}
\|(F(v))^\epsilon-F(v^\epsilon)\|_{L^\infty} \le   C(\|F\|_{C^2(\mathcal M)})\|v\|^2_{C^{0,\alpha}(\bR^{d+1})}
\epsilon^{2\alpha}\,.
\end{equation}
\end{lemma}

\begin{proof}
First observe that if $F$ is an affine map one has:
\begin{equation}
(F(v))^\epsilon-F(v^\epsilon)=0\,.\label{trivial}
\end{equation}
Therefore, combining (\ref{trivial}) with the  Taylor formula applied to  $F(v(x-y))$, viewed as a function of $y$, about $v_\epsilon(x)$  gives the following estimates:
\begin{equation}\label{basta}
\begin{aligned}
&|(F(v))^\epsilon(x)-F(v ^\epsilon(x))|=|\Big(\int_{\bR^d_y}F(v(x-y))\rho_\epsilon(y) dy\Big)-F(v^\epsilon(x))| =\\
&| \int_{\bR^d_y}(F(v(x-y)) -F(v^\epsilon(x))\rho_\epsilon(y)dy|=\\
&
 | \int_{\bR^d_y}\Big(\int_0^1\Big(\nabla_v^2 F(s v(x-y) +(1-s)v^\epsilon(x))\Big)(1-s)ds\Big)(v(x-y)  - v^\epsilon(x))^{(2)}\rho_\epsilon(y)dy| \le\\
& \|F\|_{C^2(\mathcal M))}\int_{\bR^d_y}|v(x-y) - v^\epsilon(x)|^2\rho_\epsilon(y)dy = \\
&\|F\|_{C^2(\mathcal M))} \int_{\bR^d_y} |v(x-y) - \int_{\bR^d_z} v(x-z) \rho_\epsilon(z) dz|^2\rho_\epsilon(y)dy = \\
&\|F\|_{C^2(M))}\int_{\bR^d_y}|\int_{\bR^d_z}\Big(v(x-y) -  v(x-z)\Big) \rho_\epsilon(z) dz|^2\rho_\epsilon(y)dy\,.
\end{aligned}
\end{equation}
 Since  $x\mapsto\rho_\epsilon(x)$ is equal to zero for $|x|>\epsilon\,$ then in the last term of (\ref{basta}) one has to restrict oneself to the values when $|y|<\epsilon$ and $|z|<\epsilon\,.$ This in turn implies that   $|(x-y)-(x-z)|\le 2\epsilon\,. $ Consequently one has
  \begin{equation*}
  \begin{aligned}
 &\|(F(v))^\epsilon-F(v^\epsilon)\|_{L^\infty} \le \|F\|_{C^2(\mathcal M))} \int_{\bR^d_y}|\int_{\bR^d_z}(2\epsilon)^\alpha \|v\|_{C^{0,\alpha}(\bR^{d+1})} \rho_\epsilon(z) dz|^2\rho_\epsilon(y)dy\\
&\quad =\epsilon^{2\alpha}\|F\|_{C^2(\mathcal M)) } \|v\|_{C^{0,\alpha}(\bR^{d+1})}^2\,.
\end{aligned}
\end{equation*}
 \end{proof}
 \subsection{Sufficient conditions for global entropy conservation}
 Theorem \ref{new1} can obviously be applied to the case of conservation of energy.
 Consider as above a weak solution  $u$ of the following conservation law:
 \begin{equation}
 \del_t A_0(u ) +\sum_{1\le i\le d}  \del_{x_i} (A_i(u))=0 \label{con1}\,,
\end{equation}
and assume that this equation has an extra conservation law $u\mapsto  \eta(u)$ (or entropy as usually called)   with corresponding fluxes $q_0(u)= \eta (u)$ and $q_j(u)\,,$ for $j=1,2,\ldots d\,,$ satisfying
\begin{equation}
\begin{aligned}
&\nabla_u \eta (u)\cdot \nabla_u A_0(u)=\nabla_u \eta(u)\,, \quad \hbox{and}  \\
&\nabla_u \eta (u)\cdot \nabla_u  A_j(u)= \nabla_u q_j(u)\,, \quad \hbox{for}\,\, j=1,2,\ldots d\,. \label{entropy}
\end{aligned}\, \end{equation}
Consequently, the above  gives formally the extra conservation law:
\begin{equation}
\del_t\eta(u) + \sum_{1\le i \le d} \del_{x_i}q_i(u)=0.
\end{equation}
Then, applying Theorem {\ref{new1} with, $x_0=t $ and $(B(u)=\nabla_u \eta(u))$ one has the following:
 \begin{theorem}
  Suppose that $u\in L^\infty( Q)$  is defined in a time cylindrical domain $Q=(T_1,T_2) \times \Omega\,.$ Suppose also  that $\Omega$ is a bounded open set with a Lipshitz  boundary $\del \Omega$ . This implies, in particular, the existence of   $\delta_0>0$ such that for every $x\in \Omega$ with  $d(x,\del\Omega)<\delta_0$ the function $x\mapsto d(x,\del\Omega)$ is Lipschitz and that there exists a unique point $ \hat x \in \del\Omega\,,$ depending on $x\,,$ such that
 \begin{equation}
 d(x,\del\Omega) = |x-\hat x|\quad \hbox{and} \quad \nabla_x d(x,\del\Omega) = -\vec n(\hat x) \,.
 \end{equation}

 Suppose that $u$ is a weak solution of (\ref{con1}) with the following properties:

 1. For any $\tilde Q \subset\subset Q$ there exists $\alpha>\frac13\,,$ which might depend on $\tilde Q\,,$ such that $u  \in C^{0,\alpha} (\tilde Q)\,.$

 2. Let $\delta \in (0,\frac{\delta_0}{2})$, denote by   $\Omega_\delta = \Omega \cap\{ d(x,\del\Omega)<\frac{\delta}{2} \})$ and by $Q_\delta = (T_1,T_2) \times \Omega_\delta$. Then
  \begin{equation}\label{boundary-flux}
\lim_{\delta\rightarrow 0} \sup_{(t,x)\in Q_\delta} | \sum_{1\le i \le d } q_i(u(t,x)) \vec n_i(\hat x) |= 0  \,.
\end{equation}
Then the solution $u$ conserves the total entropy $\eta(u)\,,$ i.e.,  for every $(t_1,t_2)$ satisfying $T_1<t_1\le t_2<T_2$
\begin{equation}
\int_\Omega \eta (u(t_1,x))dx= \int_\Omega \eta (u(t_2,x))dx\,.
\end{equation}
\end{theorem}
\begin{proof}
Theorem  \ref{new1}   implies that $u$ satisfies the entropy relation (\ref{conservation}) in the sense of distribution. Hence one considers a test function in $\mathcal D(Q)$ of the following form
$\phi(t,x)= \theta(t)\times \chi \Big(\frac{d(x,\del\Omega)}\delta\Big)\,.$  Here  $\chi \in C^\infty([0,\infty))$ is a cutoff function $\chi : [0,\infty) \to [0,1]$ satisfying  $\chi(s)=0$ for $s\in [0,\frac14]$ and   $\chi(s)=1$ for $s\in [\frac12,\infty)\,;$ and $\theta \in \mathcal {D} ((T_1,T_2))\,.$  Parameter  $\delta$ is chosen to be small enough such that $\mathrm {supp}\, \theta \subset (T_1+ \delta, T_2 - \delta)\,,$ and  $\delta \in(0,\frac{\delta_0}{2})\,.$

Then one has:
\begin{equation}
\begin{aligned}
& 0=\langle \phi ,\del_t\eta(u) + \sum_{1\le i \le d}\del_{x_i}  q_i(u) \rangle_{t,x}= - \int_{Q} \eta ( u(t,x)) \chi\Big (\frac{d(x,\del\Omega) }\delta\Big) \frac{d}{dt}\theta(t) dxdt
 \\
 &\quad \quad\quad\quad\quad\quad\quad\quad\quad-\int_{Q_\delta} \theta (t)\Big ( \sum_{1\le i \le d } q_i(u(t,x)) \vec n_i(\hat x)  \frac 1\delta \chi'((\frac{d(x,\del\Omega)}\delta)\Big )dxdt\,.
\end{aligned}
\end{equation}
Letting  $\delta\rightarrow 0\,,$ then by the Lebesque Dominant Convergence theorem  one obtains first the trivial relation
\begin{equation}
\lim_{\delta \to 0} \int_Q \eta ((u(t,x) ))\chi(\frac{d(x,\del\Omega) }\delta) \frac{d}{dt}\theta(t) dxdt\  = \int_{T_1}^{T_2} \Big(\frac{d}{dt}\theta(t) \int_\Omega \eta(u(t,x))dx\Big) dt\,.
\end{equation}
While for  the term
\begin{equation}
 -\int_{Q_\delta} \theta (t) \Big (\sum_{1\le i \le d } q_i(u(t,x)) \vec n_i(\hat x)  \frac 1\delta \chi'\Big(\frac{d(x,\del\Omega)}\delta\Big )\Big )dxdt\,,
\end{equation}
one uses the estimate:
\begin{equation}
\begin{aligned}
&\Big |\int_{Q_\delta} \theta (t) \Big (\sum_{1\le i \le d } q_i(u(t,x)) \vec n_i(\hat x)  \frac 1\delta \chi'\Big(\frac{d(x,\del\Omega)}\delta\Big)\Big) dxdt\Big |\le\\
&\sup_{(t,x) \in Q_\delta}  \Big| \sum_{1\le i \le d } q_i(u(t,x)) \vec n_i(\hat x)\Big| ~ \Big (\int_{\Omega_\delta} \frac 1\delta \Big |\chi'\Big(\frac{d(x,\del\Omega)}\delta\Big)\Big|dx\Big ) \times \int_{T_1}^{T_2} |\theta (t))| dt \\\le
&C \sup_{(t,x)\in Q_\delta} \Big | \sum_{1\le i \le d } q_i(u(t,x)) \vec n_i(\hat x) \Big|\,.
\end{aligned}
\end{equation}
Thanks to (\ref{boundary-flux}) the right-hand side in the inequality above tends to zero, as $\delta\rightarrow 0\,.$
Hence, from all the above one infers that for every $\theta \in \mathcal D(T_1,T_2)\,,$
\begin{equation}
\int_{T_1}^{T_2} \Big ( \frac{d}{dt}\theta(t) \int_\Omega \eta(u(t,x))dx\Big ) dt =0 \quad \hbox{or} \quad
\frac{d}{dt}  \int_\Omega \eta(u(t,x))dx=0 \quad \hbox {in} \quad \mathcal D'(T_1, T_2).
 \end{equation}
The above implies that $ \int_\Omega \eta(u(t,x))dx = \mathrm {const.}\,,$  for every $t \in (T_1, T_2)$, since $ \int_\Omega \eta(u(t,x))dx$ is a continuous function for all $t \in (T_1, T_2)$. This concludes our proof.
\end{proof}
\begin{remark} Observe that the above theorem can be applied to the incompressible Euler equations
\begin{equation}
\begin{aligned}
&\del_t v + \nabla\cdot (v\otimes v) +\nabla p=0\,,\\
&\nabla \cdot v=0\,.
\end{aligned}
\end{equation}
where $u$ in the above theorem is the column vector $u= \begin{pmatrix} v \\ p\end{pmatrix}\,,$  $\eta(u)=\frac{|v|^2}2\,,$  $A_0(u) = \begin{pmatrix} v  \\ 0   \end{pmatrix}\,, $ and   $A_i(u)=\begin{pmatrix} v_i v + \begin{pmatrix} p\\ p\\p \end{pmatrix} \\ v_i \end{pmatrix}$, for $i=1,2,3\,.$

In the forthcoming paper \cite{BGSTW} we will also provide some additional physical examples for which Theorem \ref{new1} can be applied.
\end{remark}

\begin{remark} We remark that  Lemma \ref{basic-lemma} could be generalized in the spirit of the Lemma 2.1  of \cite{ContiDelellisSzekzlyhidi}  to give
\begin{equation}
\| F(v^\epsilon) -(F(v))^\epsilon\|_{C^r} \le C \|F\|_{C^{r+2}}\|v\|^2_{C^{r+\alpha}} \epsilon^{2\alpha}
\end{equation}
\end{remark}

 \section*{Acknowledgements}
EST would like to thank the \'{E}cole Polytechnique for its kind hospitality, where this work was completed, and the \'{E}cole Polytechnique Foundation for its partial financial support through the 2017-2018 ``Gaspard Monge Visiting Professor" Program. EST was also supported in part by  the ONR grant N00014-15-1-2333, and by the Einstein Stiftung/Foundation - Berlin, through the Einstein Visiting Fellow Program. This work was partially supported by the Simons - Foundation grant 346300 and the Polish Government MNiSW 2015-2019 matching fund.  PG and A\'SG. received support from the National Science Centre (Poland),
2015/18/MST1/00075.

\end{document}